\documentclass[a4paper,11pt]{amsart}

\usepackage{multicol}
\usepackage{amsmath,latexsym,amsbsy,amssymb}
\usepackage{enumerate}
\usepackage{amsthm}
\usepackage[latin1]{inputenc}

\newtheorem{Theorem}{Theorem}[section]
 
\newtheorem{Proposition}[Theorem]{Proposition}

\newtheorem{Remark}[Theorem]{Remark}

\numberwithin{equation}{section}

\newcommand{\R}{{\mathbb R}}
\newcommand{\N}{{\mathbb N}}

\begin{document}

\title{Variational Analysis for the Bilateral Minimal Time Function}
\author[Luong V. Nguyen]{Luong V. Nguyen }
\address{Institute of Mathematics, Polish Academy of Sciences, ul. \'Sniadeckich 8, 
00-656  Warsaw, Poland}
\email{vnguyen@impan.pl; luonghdu@gmail.com}



\date{\today}
\begin{abstract}
In this paper, we derive formulas for the Fr\'echet (singular) subdiferentials of  the bilateral minimal time function $T:\mathbb{R}^n \times \mathbb{R}^n \to [0,+\infty]$  associated with a system governed by differential inclusions. As a consequence, we  give a connection between the Fr\'echet normals to the sub-level sets of $T$ and to its epigraph. Finally, we  show that the Fr\'echet normal cones to the sub-level set of $T$ at a point $(\alpha,\beta)$ and to epi($T$) at $((\alpha,\beta),T(\alpha,\beta))$ have the same dimension. 
\end{abstract}
\maketitle
\section{Introduction}
Let $F:\mathbb{R}^n  \rightrightarrows \mathbb{R}^n$ be a multifunction. We consider the system governed by the differential inclusion associated with $F$:
\begin{equation}
\label{SDI}
\left\{ \begin{array}{lcl}
    \dot{x}(t) & \in &F(x(t)),\quad\quad \mathrm{a.e.}\, t>0. \\
    x(0)& = & x_0 \in \R^n
  \end{array}\right.
\end{equation}
 A trajectory starting at $x_0$ of $F$ is a solution of the differential inclusion (\ref{SDI}), i.e., an absolutely continuous function $x: [0,+\infty) \to \R^n$ satisfying $\dot{x}(t) \in F(x(t))$ for a.e. $t>0$ and $x(0) = x_0$. 

The bilateral minimal time function  $T:\mathbb{R}^n \times \mathbb{R}^n \to [0,+\infty]$ associated with (\ref{SDI}) is defined as follows: for each pair $(\alpha,\beta) \in \mathbb{R}^n \times \mathbb{R}^n$, $T(\alpha,\beta)$ is the minimal time taken by  the trajectories  of $F$ starting at the point  $\alpha$  to reach the point  $\beta$. If no trajectory starting at $\alpha$ can reach $\beta$, then $T(\alpha,\beta) = +\infty$. When we fix the final point $\beta$, we get the function $T(\cdot,\beta)$ - the well known unilateral minimal time function associated to the target $\{\beta\}$. This function  is a classical and widely studied topic in control theory (see, e.g. \cite{CS95,CSB04,CAK,CMW,GK,CKL,CL,CW1,CW2,NS3,LN,S04,WY98} and references mentioned therein).

The bilateral minimal time function $T$ was introduced in \cite{CN04} by Clarke and Nour to study the Hamilton - Jacobi equation of optimal time problems in a domain containing the target. In that paper, using the function $T$, the authors constructed  proximal solutions of the relevant equation and studied the existence of time-geodesic trajectories. After that, the bilateral minimal time function and its regularity properties were studied deeply in \cite{CN04,CN06,CN13,NS1,NS2}. In these papers, the authors generalized known results for the unilateral minimal time function to the bilateral case.

Inspired by \cite{LN} where a relationship between the \textit{proximal normal cones} to sub-level sets of the unilateral minimal time function and its epigraph is given, in the present paper,  we give a similar relationship between the \textit{Fr\'echet normal cones} to sub-level sets and the epigraph of the bilateral minimal time function. Our main result is presented in Theorem \ref{RE}. By using this result, we prove a special feature of the minimal time function - evidently not true for a general, even convex, function - that is : for $(\alpha,\beta) \in \R^n \times \R^n$ with $0< T(\alpha,\beta) < +\infty$, the Fr\'echet normal cone to the epigraph of $T$ at $((\alpha,\beta),T(\alpha,\beta))$ has the same (algebraic) dimension of the Fr\'echet normal cone to the sub-level set $\{(x,y) \in \R^n\times \R^n: T(x,y) \le T(\alpha,\beta)\}$ at the point $(\alpha,\beta)$.

It is worth mentioning that  the proof of Theorem \ref{RE} relies heavily on the representations of the Fr\'echet (singular) subdifferentials of $T$ (see Theorem \ref{PN}  and Theorem \ref{HPN}) and the following interesting property of normal vectors to the sub-level sets of $T$: for $(\alpha,\beta) \in \mathbb{R}^n \times \mathbb{R}^n$ with $0<T(\alpha,\beta) < +\infty$, if $(\zeta,\theta)$ belongs to the Fr\'echet normal cone of the set  $ \{(x,y)\in \mathbb{R}^n\times \mathbb{R}^n: T(x,y) \le T(\alpha,\beta)\}$ at $(\alpha,\beta)$, then
$$h(\alpha,\zeta) = h(\beta,-\theta),$$
where $h:\mathbb{R}^n \times \mathbb{R}^n \to \mathbb{R}$, the Hamintonian associated to $F$, is defined by
$$h(x,p) = \min_{v\in F(x)} \langle v,p\rangle,\qquad (x,p) \in \mathbb{R}^n \times \mathbb{R}^n.$$ 

The paper is organized as follows. In Section 2, we give some notions, definitions and preliminaries which will be used in the sequel. Section 3 is devoted to variational analysis for the bilateral minimal time function.
\section{Preliminaries}
\subsection{Notations and basic facts}
In this section we recall some basic concepts of nonsmooth analysis. Standard references are in \cite{CLSW,RW}.\\ 
We denote by $||\cdot||$ the Euclidean norm in $\R^n$,  by $\langle\cdot,\cdot\rangle$ the inner product. We also denote by $B(x,r)$ the open ball of radius $r>0$ centered at $x$, and $\mathbb{S}^{n-1}$ the unit sphere in $\R^n$. We will use the shortened $B = B(0,1)$. For any subset $E$ of $\R^n$, we denote by $\mathrm{bdry}E$ its boundary, by $\bar{E}$ its closure and by $\mathrm{Proj_E(x)}$ the projection of $x\in \R^n$ on $E$.  A subset $C$ of $\R^n$ is called a cone if and only if $\lambda x \in C$ for any $x\in C$ and $\lambda \ge 0$. We say that $\kappa \in \N$ is the dimension of a cone $C$ if there exist $v_1,\cdots,v_\kappa \in C$ such that they are linearly independent and for any $v\in C$ there exist nonnegative numbers $\lambda_1,\cdots,\lambda_\kappa$ such that $v = \lambda_1 v_1 +\cdots + \lambda_\kappa v_\kappa$.

 Let $S\subset \R^n$ be a closed set and let $x\in S$. The \textit{Fr\'echet normal cone} to $S$ at $x$, written $\widehat{N}_S(x)$, is the set
$$\widehat{N}_S(x):= \left\{\zeta \in \R^n: \limsup_{S \ni y \to x} \frac{\langle \zeta, y-x\rangle}{||y-x||} \le 0\right\}.$$ 
Elements in $\widehat{N}_S(x)$ are called \textit{Fr\'echet normals} to $S$ at $x$.
 
 In other words, $\zeta \in \widehat{N}_S(x)$ if and only if  for any $\varepsilon >0$, there exists $\delta >0$ such that 
 $$\langle\zeta, y-x\rangle \le \varepsilon ||y-x||,\qquad \forall y\in B(x,\delta).$$

Let $f: \R^n \to \R\cup \{+\infty\}$ be an extended real-valued function. The effective domain of $f$ is the set $\mathrm{dom}(f):= \{x\in \R^n: f(x) < +\infty\}$ and the epigraph of $f$ is the set $\mathrm{epi}(f): =\{ (x,\alpha)\in \R^n\times \R: x\in \mathrm{dom}(f),\alpha \ge f(x)\}$. We say that $f$ is \textit{lower semicontinuous} at $x_0\in \R^n$ if for every $\varepsilon >0$, there exists a neighborhood $V$ of $x_0$ such that $f(x)\ge f(x_0) - \varepsilon$ for all $x\in V$ when $f(x_0) <+\infty$ and $f(x)$ tends to $+\infty$ as $x$ tends to $x_0$ when $f(x_0) =+\infty$. Equivalently,
$$\liminf_{x\to x_0} f(x) \ge f(x_0).$$
We say $f$ is lower semicontinuous  if $f$ is lower semicontinuous at every $x_0\in \R^n$. Observe that if $f$ is lower semicontinuous then its sub-level sets are closed.

Let $x\in \mathrm{dom}(f)$. The \textit{Fr\'echet subdifferential} of $f$ at $x$ is the set
$$\widehat{\partial} f(x): = \left\{ \zeta \in \R^n: \liminf_{y\to x} \frac{f(y)-f(x) - \langle \zeta,y-x\rangle}{||y-x||} \ge 0   \right\}.$$
 In other words, $\zeta \in \widehat{\partial} f(x)$ if and only if for any $\varepsilon >0$, there exists $\delta >0$ such that
 $$\langle \zeta,y-x\rangle \le f(y) - f(x) + \varepsilon ||y-x||,\qquad \forall y\in B(x,\delta).$$
 The Fr\'echet subdifferential of $f$ at $x$ can also  be defined as follows:
 $$\widehat{\partial}f(x) = \left\{ \zeta \in \R^n: (\zeta,-1) \in \widehat{N}_{\mathrm{epi}(f)}(x,f(x))\right\}.$$
 Elements in $\widehat{\partial}f(x) $ are called \textit{Fr\'echet subgradients} of $f$ at $x$.\\
 The \textit{Fr\'echet singular subdifferential} of $f$ at $x$ is the set
 $$\widehat{\partial}^\infty f(x): = \left\{ \zeta \in \R^n: (\zeta,0) \in \widehat{N}_{\mathrm{epi}(f)} (x,f(x)) \right\}.$$
 In other words, $\zeta \in \widehat{\partial}^\infty f(x)$ if and only if for any $\varepsilon >0$, there exists $\delta >0$ such that
 $$\langle \zeta, y-x\rangle\le \varepsilon (||y-x|| + |\beta - f(x)|), \qquad\forall  y\in B(x,\delta),\, (y,\beta)\in \mathrm{epi}(f).$$
 Elements in $\widehat{\partial}^\infty f(x)$  are called \textit{Fr\'echet singular subgradients} of $f$ at $x$.
\subsection{The bilateral minimum time function}
Let $F:\R^n \rightrightarrows \R^n$ be a multifunction. In this paper, we require the following assumptions on the multifunction $F$.
\begin{itemize}
\item[(F1)] $F(x)$ is a nonempty compact convex set for all $x\in \R^n$.
\item[(F2)] $F$ is locally Lipschitz, i.e., for any compact set $K$, there exists a constant $L := L(K)$ such that
$$F(x) \subset F(y) + L||y-x||\bar{B},\quad\forall x,y \in K.$$
\item[(F3)] There exist some positive constants $\gamma$ and $c$ such that for all $x\in \R^n$,
$$v\in F(x) \Rightarrow ||v|| \le \gamma ||x|| + c.$$
\end{itemize}
For some $\tau>0$, we consider the differential inclusion
\begin{equation}
\label{DI}
\left\{ \begin{array}{lcl}
    \dot{x}(t) & \in &F(x(t)),\quad\quad \mathrm{a.e.}\, t\in [0,\tau]. \\
    x(0)& = & x_0 \in \R^n
  \end{array}\right.
\end{equation}
A solution of (\ref{DI}) is an absolutely continuous function $x(\cdot)$ defined on $[0,\tau]$ with the initial condition $x(0) = x_0$. We call $x(\cdot)$ a trajectory of $F$ starting at $x_0$. 

Notice that, under our assumptions on $F$, if $x(\cdot)$ is a trajectory of $F$ defined on $[0,\tau]$ then by Gronwall's Lemma, there exists a constant $M>0$ such that $||x(t) -x_0||\le Mt$ for all $t\in [0,\tau]$. In this paper, for simplicity, we fix the constant $M$ for all $\tau>0$ and for all trajectories. The following theorem gives some information regarding $C^1$ trajectories of $F$ which will be useful in the sequel.
\begin{Theorem}\cite{WY98} \label{C1}
Assume (F1)-(F3). Let $E\subset \R^n$ be compact. Then there exists $\tau>0$ such that associated to every $x\in E$ and $v\in F(x)$ is a trajectory $x(\cdot)$ defined on $[0,\tau]$ with $\dot{x}(0) = v$. Moreover, for all $t\in [0,\tau]$, we have $||\dot{x}(t) - v|| \le Kt$, for some constant $K>0$ independent of $x$
\end{Theorem}
The bilateral minimal time function  $T: \R^n \times \R^n \to [0,+\infty]$ is defined as follows: for $(\alpha,\beta) \in \R^n \times \R^n$,
\begin{equation}
\label{BMTF}
T(\alpha,\beta) = \inf\{T\ge 0: \text{there exists some trajectory}\,x(\cdot)\,\,\text{of}\,\, F\,\,\text{with}\,\,x(0) = \alpha\,\,\text{and}\,\,x(T) =\beta \}.
\end{equation}
If there is no trajectory steering $\alpha$ to $\beta$, then $T(\alpha,\beta) =+\infty$. It may happen that, for $\alpha,\beta\in \R^n$, $T(\alpha,\beta) <+\infty$ and $T(\beta,\alpha) = +\infty$. Obviously, $T(\alpha,\alpha) = 0$, for all $\alpha \in \R^n$.\\
We have following properties of the bilateral minimal time function $T$ (see \cite{CN06}):
\begin{itemize}
\item $T$ is lower semicontinuous.
\item If $T(\alpha,\beta) <+\infty$, then the infimum in (\ref{BMTF}) is attained.
\item For all $\alpha,\beta,\gamma\in \R^n$, we have the following triangle inequality
\end{itemize}
$$T(\alpha,\beta) \le T(\alpha,\gamma) + T(\gamma,\beta).$$
For $t>0$, the set $\mathcal{R}(t): = \{(\alpha,\beta) \in \R^n\times \R^n: T(\alpha,\beta) \le t\}$ is called the reachable set at time $t$ and the set
$$\mathcal{R} := \bigcup_{t\ge 0}\mathcal{R}(t) = \{(\alpha,\beta) \in \R^n\times \R^n: T(\alpha,\beta) <+\infty\},$$
is called the reachable set.

\section{Variational analysis for the bilateral minimal time function}
The following theorem presents a formula for the Fr\'echet sudifferential of the bilateral minimal time function at a point $(\alpha,\alpha)\in \R^n \times \R^n$. The formula is similar to the one for the proximal subdifferential given in \cite{CN06} (see Theorem 4.10 (1) in \cite{CN06}).
\begin{Theorem}
We have
\begin{equation}
\label{1}
\widehat{\partial}T(\alpha,\alpha) = \{(\zeta,-\zeta) \in \R^n \times \R^n: h(\alpha,\zeta) \ge -1\},
\end{equation}
for any $\alpha\in \R^n$.
\end{Theorem}
\begin{proof}
Let $\alpha \in \R^n$ and $(\zeta,\theta) \in \widehat{\partial}T(\alpha,\alpha)$. Then for any $\varepsilon >0$, there exists $\sigma >0$ such that 
\begin{equation}
\label{2}
\langle (\zeta,\theta), (x,y) - (\alpha,\alpha)\rangle - T(x,y) \le \varepsilon ||(x,y) - (\alpha,\alpha)||,
\end{equation}
for all $(x,y) \in B((\alpha,\alpha),\sigma)$.

Taking $x = y$ in (\ref{2}), we have, for all $x\in B(\alpha,\sigma/2)$, that
\begin{equation}
\label{3}
\langle (\zeta,\theta), (x -\alpha,x-\alpha)\rangle  \le \varepsilon ||(x-\alpha,x-\alpha)||,
\end{equation}
Let now $w\in \R^n$ and set $x_n := \alpha+ w/n$ for $n\in \N^*$. Then there is some $n_0 >0$ such that for all $n\ge n_0$ we have $x_n \in B(\alpha,\sigma/2)$. Thus, in (\ref{3}), taking $x = x_n$ with $n\ge n_0$, one get 
$$\langle (\zeta,\theta),(w,w)\rangle \le \varepsilon ||(w,w)||.$$
Letting $\varepsilon \to 0+$ in the latter inequality, we obtain $\langle (\zeta,\theta),(w,w)\rangle \le 0$ for all $w\in \R^n$. This implies $\zeta = -\theta$.

Let $v\in F(\alpha)$ be such that $\langle v,\zeta\rangle = h(\alpha,\zeta)$. By Theorem \ref{C1}, there is a $C^1$ trajectory $x(\cdot)$ of $-F$ such that $x(0) = \alpha$ and $\dot{x}(0) = -v$. There exists $\delta >0$ such that $x(t) \in B(\alpha,\sigma/2)$ for all $t\in [0,\delta]$.  Observe that $T(x(t),\alpha) \le t$ for all $t\in [0,\delta]$. For $t\in [0,\delta]$, taking $x= x(t), y = \alpha$ in (\ref{2}), we have
$$\langle (\zeta,\theta),(x(t),\alpha) - (\alpha,\alpha) \rangle - t \le \varepsilon ||(x(t),\alpha) - (\alpha,\alpha)||,$$
and then
$$\langle \zeta, z(t) - \alpha\rangle \le t + \varepsilon M t.$$
Dividing both sides of the latter inequality by $t>0$ then letting $t\to 0+$, we get
$$\langle \zeta, \dot{z}(0)\rangle \le 1 + \varepsilon M.$$
Since $\varepsilon >0$ is arbitrary, we have $\langle \zeta, v\rangle = \langle \zeta, -\dot{z}(0)\rangle \ge -1$. That is $h(\alpha,\zeta) \ge -1$.

Now let $(\alpha,\zeta) \in \R^n \times \R^n$ be such that $h(\alpha,\zeta) \ge -1$. We want to show that $(\zeta,-\zeta) \in \widehat{\partial}T(\alpha,\alpha)$. Assume to the contrary that $(\zeta,-\zeta) \not \in \widehat{\partial} T(\alpha,\alpha)$. Then there exist a consant $C>0$ and a sequence $\{(\alpha_n,\beta_n)\}$ such that $(\alpha_n,\beta_n) \to (\alpha,\alpha)$, $(\alpha_n,\beta_n) \ne (\alpha,\alpha)$ and 
\begin{equation}
\label{4}
\langle (\zeta,-\zeta), (\alpha_n-\alpha, \beta_n - \alpha)\rangle - T(\alpha_n,\beta_n) > C ||(\alpha_n-\alpha,\beta_n- \alpha)||,\qquad \forall n.
\end{equation}
It follows from (\ref{4}) that for all $n$
\begin{equation}
\label{01}
0< T_n: = T(\alpha_n,\beta_n) \le 2||\zeta||. ||(\alpha_n-\alpha,\beta_n-\alpha)|| <+\infty.
\end{equation}
Thus, for each $n$, there exists a trajectory $x_n(\cdot)$ of $F$ such that $x_n(0) = \alpha_n$ and $x_n(T_n) = \beta_n$.

We have, for all $n$ and all $t\in [0,T_n]$, that
$$||x_n(t)-\alpha|| \le ||x_n(t) - \alpha_n|| + ||\alpha_n - \alpha|| \le MT_n + ||\alpha_n - \alpha||,$$
and then
$$||x_n(t)-\alpha|| \le  MT_n + ||(\alpha_n - \alpha,\beta_n -  \alpha)||.$$
Let $y_n(t) := \mathrm{Proj}_{F(\alpha)} (\dot{x}_n(t))$ on $[0,T_n]$. By the Lipschitz continuity of $F$, we have, for all $n$, 
$$||y_n(t) - \dot{x}_n(t)|| \le L||x_n(t)  - \alpha||,\quad \forall t\in [0,T_n].$$
 Moreover, since $h(\alpha,\zeta) \ge -1$, we have $\langle \zeta,y_n(t)\rangle \ge -1$ for all $n$ and for all $t\in [0,T_n]$. Then using (\ref{01}),
\begin{eqnarray*}
\langle (\zeta,-\zeta), (\alpha_n-\alpha, \beta_n -  \alpha)\rangle - T(\alpha_n,\beta_n) &=& \langle \zeta,\alpha_n - \beta_n\rangle - T_n\\
&\le& \langle \zeta,\int_0^{T_n} \dot{x_n}(t)dt\rangle  - \int_0^{T_n} \langle \zeta, y_n(t)\rangle dt\\
&\le& L||\zeta|| \int_0 ^{T_n} ||y_n(t) - \dot{x}_n(t)|| dt\\
&\le& L ||\zeta|| (MT_n^2 + ||(\alpha_n-\alpha,\beta_n -  \alpha)|| T_n)\\
&\le& 2L||\zeta||^2 (M||\zeta||+1)||(\alpha_n-\alpha,\beta_n -  \alpha)||^2.
\end{eqnarray*}
Combining with (\ref{4}) we have, for all $n$, that
$$C ||(\alpha_n-\alpha,\beta_n- \alpha)|| < 2L||\zeta||^2 (M||\zeta||+1)||(\alpha_n-\alpha,\beta_n -  \alpha)||^2.$$
Dividing both sides of the latter inequality by $||(\alpha_n-\alpha,\beta_n- \alpha)|| >0$ then letting $n\to\infty$, we obtain $ C\le 0$. This is a contradiction. Therefore $(\zeta,-\zeta)\in \widehat{\partial} T(\alpha,\alpha)$.
\end{proof}
We can also derive a formula for the Fr\'echet subdifferential of the bilateral minimal time function at a point $(\alpha,\beta)\in \R^n\times \R^n$ with $\alpha \ne \beta$. Again, the formula is similar to the one for the proximal subdifferential given in \cite{CN06} (see Theorem 4.10 (2) in \cite{CN06}). Before stating the result, in the next proposition, we present a characterisation of a Fr\'echet normals to sub-level sets of the bilateral minimal time function which will be useful in the sequel
\begin{Proposition}\label{eqH}
Let $(\alpha,\beta)\in \R^n\times \R^n$ be such that $0<r:=T(\alpha,\beta) <\infty$. If $(\zeta,\theta) \in \widehat{N}_{\mathcal{R}(r)}(\alpha,\beta)$ then $h(\alpha,\zeta) = h(\beta,-\theta) \le 0$.
\begin{proof}
Since $(\zeta,\theta)\in  \widehat{N}_{\mathcal{R}(r)}(\alpha,\beta)$, for any $\varepsilon >0$, there exists $\sigma >0$ such that, for all $(x,y)\in\mathcal{R}(r) \cap B((\alpha,\beta),\sigma)$, one has
\begin{equation}
\label{E1}
\langle (\zeta,\theta),(x,y) - (\alpha,\beta)\rangle \le \varepsilon ||(x,y) - (\alpha,\beta)||.
\end{equation}
Let $x(\cdot)$ be a trajectory of $F$ such that $x(0) =\alpha$ and $x(r) = \beta$. Then for all $t\in [0,r]$ sufficiently small, we have $(x(t),\beta) \in \mathcal{R}(r)\cap B((\alpha,\beta),\sigma)$. Hence, by (\ref{E1}), we have, for  $t>0$ sufficiently small, that
$$\langle (\zeta,\theta),(x(t),\beta)-(\alpha,\beta)\rangle \le \varepsilon ||(x(t),\beta)-(\alpha,\beta)||.$$
Thus,
\begin{equation}
\label{E2}
\int_0^t \langle \zeta,\dot{x}(s)\rangle ds \le \varepsilon||x(t) - x(0)||\le  \varepsilon Mt.
\end{equation}
Let $g(\cdot)$ be the  projection of $\dot{x}(\cdot)$ on $F(\alpha)$ restricted to $[0,r]$. Then by the Lipschitz continuity of $F$, 
\begin{equation}
\label{E3}
||\dot{x}(s) - g(s)|| \le L||x(s) - \alpha|| \le LMs,\quad\text{for all}\,\,s\in [0,r].
\end{equation}
Using (\ref{E2}) - (\ref{E3}), we have, for $t\in [0,r]$ sufficiently small, that
\begin{eqnarray}
\label{E4}
h(\alpha,\zeta)t &\le& \int_0^t\langle \zeta,g(s)\rangle ds = \int_0^t \langle \zeta,\dot{x}(s)\rangle ds + \int_0^t \langle \zeta, g(s) - \dot{x}(s)\rangle ds \nonumber\\
&\le& \varepsilon Mt + ||\zeta||\int_0^t ||g(s)-\dot{x}(s)||ds \le \varepsilon Mt + LM||\zeta||t^2.
\end{eqnarray}
Dividing (\ref{E4}) by $t>0$ then letting $t\to 0+$, we get $h(\alpha,\zeta) \le \varepsilon M$. Since $\varepsilon >0$ is arbitrary, we conclude that $h(\alpha,\zeta) \le 0$. 

Let $p(\cdot)$ be the projection of $\dot{x}(\cdot)$ on $F(\beta)$ restricted to $[0,r]$. We have
\begin{equation} 
\label{E5}
||\dot{x}(s)-p(s)|| \le L||x(s) - \beta|| \le ML(r-s),\quad\forall s\in [0,r].
\end{equation}

Now let $w\in F(\alpha)$ be such that
$$\langle w,\zeta\rangle = h(\alpha,\zeta) = \min_{v\in F(\alpha)}\langle v,\zeta\rangle.$$
By Theorem \ref{C1}, there exist $\tau>0$ and a $C^1$ trajectory $z(\cdot)$ of $-F$ on $[0,\tau]$ with $z(0) = \alpha$ and $\dot{z}(0) = -w$. For $t\in [0,\tau]$, set $q(t) = z(\tau-t)$. Then $q(\cdot)$ is a trajectory of $F$ with $q(0) = z(\tau)$ and $q(\tau) = \alpha$. By the principle of optimality, we have, for all $t\in [0,\tau]$, that
$$T(z(t),\alpha) = T(q(\tau-t),\alpha)  = T(q(\tau-t),q(\tau)) \le t.$$
Fixed $0<t<\min\{r,\tau\}$. By the triangle inequality, we have $(z(t),x(r-t))\in \mathcal{R}(r)$. We may choose $\tau>0$ sufficiently small such that $(z(t),x(r-t))\in \mathcal{R}(r) \cap B((\alpha,\beta),\sigma)$. It follows from (\ref{E1}) that
$$\langle (\zeta,\theta), (z(t),x(r-t)) - (\alpha,\beta)\rangle \le \varepsilon ||(z(t),x(r-t)) - (\alpha,\beta)||.$$
and then
\begin{equation}
\label{E6}
\langle \zeta,z(t) - z(0)\rangle + \int_{r-t}^r \langle -\theta,\dot{x}(s)\rangle ds \le 2\varepsilon Mt.
\end{equation}
From (\ref{E5}) and (\ref{E6}), for $0<t<\min\{r,\tau\}$, one has
\begin{eqnarray} \label{E7}
\langle \zeta, z(t) - z(0) \rangle + h(\beta,-\theta)t &\le& \langle \zeta, z(t) - z(0) \rangle + \int_{r-t}^r \langle -\theta,\dot{x}(s)\rangle ds + \int_{r-t}^r \langle -\theta,p(s) -\dot{x}(s)\rangle ds\nonumber\\
&\le& 2\varepsilon Mt +||\theta||\int_{r-t}^r||p(s) -\dot{x}(s)||ds \le 2\varepsilon Mt + ML||\theta||t^2.
\end{eqnarray}
Dividing (\ref{E7}) by $t>0$ then letting $t\to 0+$, we obtain
$$-h(\alpha,\zeta) +h(\beta,-\theta)= \langle \zeta, -w\rangle + h(\beta,-\theta)\le 2\varepsilon M.$$
Letting $\varepsilon \to 0+$ in the latter inequality, we get $-h(\alpha,\zeta) +h(\beta,-\theta) \le 0$, i.e., $h(\beta,-\theta) \le h(\alpha,\zeta)$.

Similarly, one can show that $h(\alpha,\zeta) \le h(\beta,-\theta).$ Thus
$$h(\alpha,\zeta) = h(\beta,-\theta) \le 0.$$
The proof is complete.
\end{proof}
\end{Proposition}
\begin{Theorem}\label{PN}
Let $(\alpha,\beta)\in \R^n\times \R^n$ be such that $0<r:=T(\alpha,\beta) <\infty$. One has
$$\widehat{\partial} T(\alpha,\beta) = \widehat{N}_{\mathcal{R}(r)} (\alpha,\beta) \cap \{(\zeta,\theta) \in \R^n\times \R^n: h(\alpha,\zeta) = h(\beta,-\theta) = -1\}.$$
\end{Theorem}
\begin{proof}
Let $(\zeta,\theta)\in \widehat{\partial} T(\alpha,\beta)$. Then for any $\varepsilon >0$ there exists $\sigma >0$ such that 
\begin{equation}
\label{5} 
\langle (\zeta,\theta), (x,y) - (\alpha,\beta)\rangle \le T(x,y) - r +\varepsilon ||(x,y) - (\alpha,\beta)||,
\end{equation}
for all $(x,y) \in B((\alpha,\beta),\sigma)$.

It follows from (\ref{5}) that
$$\langle (\zeta,\theta), (x,y) - (\alpha,\beta)\rangle \le \varepsilon ||(x,y) - (\alpha,\beta)||,\qquad \forall (x,y) \in \mathcal{R}(r) \cap B((\alpha,\beta),\sigma).$$
This means that $(\zeta,\theta) \in \widehat{N}_{\mathcal{R}(r)} (\alpha,\beta) $.

Let $z(\cdot)$ be as in  the proof of Proposition \ref{eqH}.  By the triangle inequality, we have
$$T(z(t),\beta) \le T(z(t),\alpha) + T(\alpha,\beta) \le t+r,\quad\forall t\in [0,\tau].$$
In (\ref{5}), taking $x = z(t),y = \beta$, we obtain
$$\langle \zeta,z(t) - \alpha\rangle \le  t + \varepsilon ||z(t) - \alpha|| \le t + \varepsilon Mt.$$
Dividing the latter inequality by $t>0$ then letting $t\to 0$, we  have
$$-h(\alpha,\zeta) = \langle \zeta,-w\rangle = \langle \zeta,\dot{z}(0)\rangle \le 1 + M\varepsilon.$$
Since $\varepsilon >0$ is arbitrary, we obtain 
\begin{equation}
\label{6}h(\alpha,\zeta) \ge -1.
\end{equation} 

Let $x(\cdot)$ and $g(\cdot)$ be also as in the proof of Proposition \ref{eqH}. We have that $T(x(t),\beta) = r- t$ for all $t\in [0,r]$. In (\ref{5}), taking $x = x(t), y = \beta$, one has
$$\int_0^t \langle\zeta,\dot{x}(t)\rangle = \langle (\zeta,\beta),(x(t),\beta) - (\alpha,\beta)\rangle \le -t +\varepsilon ||(x(t),\beta) - (\alpha,\beta)|| \le -t +\varepsilon Mt.$$
Then
\begin{eqnarray*}
h(\alpha,\zeta)t &\le& \int_0^t\langle \zeta,g(s)\rangle ds = \int_0^t \langle \zeta,\dot{x}(s)\rangle ds + \int_0^t \langle \zeta, g(s) - \dot{x}(s)\rangle ds \nonumber\\
&\le& -t +\varepsilon Mt + ||\zeta||\int_0^t ||g(s)-\dot{x}(s)||ds \le -t +\varepsilon Mt + LM||\zeta||t^2.
\end{eqnarray*}
This implies that $h(\alpha,\zeta) \le -1$. Together with (\ref{6}) and Proposition \ref{eqH} we get $h(\alpha,\zeta) = h(\beta,-\theta) = -1$.

Conversely, let $(\zeta,\theta) \in \widehat{N}_{\mathcal{R}(r)}(\alpha,\beta)$ with $h(\alpha,\zeta) = h(\beta,-\theta) = -1$. We attempt to show that $(\zeta,\theta) \in \widehat{\partial}T(\alpha,\beta)$. Assume to the contrary that there exist a constant $C>0$ and a sequence $\{(\alpha_n,\beta_n)\}$ such that $(\alpha_n,\beta_n) \to (\alpha,\beta)$ as $n\to \infty$ and $(\alpha_n,\beta_n)\ne (\alpha,\beta)$ and
\begin{equation}
\label{7}
 \langle (\zeta,\theta),(\alpha_n-\alpha,\beta_n-\beta)\rangle + r - T(\alpha_n,\beta_n) > C||(\alpha_n-\alpha,\beta_n-\beta)||,\qquad\forall n.
\end{equation}
For each $n$, set $T_n := T(\alpha_n,\beta_n)$ and $\Lambda_n:= ||(\alpha_n-\alpha,\beta_n-\beta)||$. There are 3 possible cases.

\textbf{Case 1}.  $T_n = r$ for infinitely many $n$. It follows from (\ref{7}) that 
$$ \langle (\zeta,\theta),(\alpha_n-\alpha,\beta_n-\beta)\rangle > C||(\alpha_n-\alpha,\beta_n-\beta)||.$$
It is evident that this contradicts to $(\zeta,\theta) \in \widehat{N}_{\mathcal{R}(r)}(\alpha,\beta)$.

\textbf{Case 2}. $T_n >r$ for infinitely many $n$. It follows from (\ref{7}) that
\begin{equation}
\label{E0}
0< T_n - r  < ||(\zeta,\theta)||.||(\alpha_n-\alpha,\beta_n-\beta)|| < +\infty.
\end{equation}
Set $d_n:= (T_n-r)/2$.  Let $x_n(\cdot)$ be a trajectory of $F$ such that $x_n(0) = \alpha_n$ and $x_n(T_n) = \beta_n$. Set $u_n: = x_n(d_n), w_n:=x_n(T_n-d_n)$. Then $T(u_n,w_n) = r$, i.e., $(u_n,w_n) \in \mathcal{R}(r)$.

For $t\in [0,T_n]$, we have that
\begin{equation}
\label{8}
||x_n(t)-\alpha|| \le ||x_n(t)-\alpha_n|| + ||\alpha_n - \alpha|| \le Mt + ||\alpha_n-\alpha||,
\end{equation}
and
\begin{equation}
\label{9}
||x_n(T_n-t)-\beta|| \le ||x_n(T_n-t)-\beta_n|| + ||\beta_n - \beta|| \le Mt + ||\beta_n-\beta||.
\end{equation}
Let $p_n(\cdot)$ and $q_n(\cdot)$ be the projections of $\dot{x}_n(\cdot)$ on $F(\alpha)$ and $F(\beta)$ restricted on $[0,T_n]$, respectively. By Lipschitz continuity of $F$ and (\ref{8}), (\ref{9}),  one has, for all $t\in [0,T_n]$, that
\begin{equation}
\label{10}
||\dot{x}_n(t)-p_n(t)|| \le L||x_n(t) - \alpha|| \le LMt + L||\alpha_n - \alpha||,
\end{equation}
and
\begin{equation}
\label{11}
||\dot{x}_n(T_n-t) - q_n(T_n-t)|| \le L||x_n(T_n-t) - \beta|| \le LMt + L||\beta_n-\beta||.
\end{equation}
Using (\ref{E0}) - (\ref{9}) and the facts  $(u_n,w_n) \in \mathcal{R}(r), (\zeta,\theta)\in \widehat{N}_{\mathcal{R}(r)}(\alpha,\beta)$, for any $\varepsilon >0$, for $n$ sufficiently large, we have
\begin{eqnarray}
\label{12}
 \langle (\zeta,\theta),(u_n,w_n)-(\alpha,\beta)\rangle &\le& \varepsilon ||(u_n-\alpha,w_n-\beta)|| \nonumber\\
 &\le& \varepsilon( Md_n + ||\alpha_n -\alpha|| + Md_n + ||\beta_n-\beta|| ) \nonumber\\
 &=& \varepsilon (2Md_n + 2||(\alpha_n-\alpha,\beta_n-\beta)|| \le \varepsilon (M||(\zeta,\theta)|| + 2) \Lambda_n.
\end{eqnarray}
Moreover, using (\ref{E0}),  (\ref{10}), (\ref{11}) and the fact  $h(\alpha,\zeta) = h(\beta,-\theta) = -1$, one has
\begin{eqnarray}
\label{13}
\langle (\zeta,\theta),(\alpha_n,\beta_n)-(u_n,w_n)\rangle &=& \langle\zeta,\alpha_n-u_n\rangle + \langle\theta,\beta_n - w_n\rangle \nonumber\\
&=& \int_0^{d_n} \langle \zeta, -\dot{x}_n(t)\rangle dt + \int_{T_n-d_n}^{T_n} \langle \theta,\dot{x}_n(t)\rangle dt\nonumber\\
&=& \int_0^{d_n} \langle \zeta, -p_n(t)\rangle dt + \int_0^{d_n} \langle \zeta, p_n(t)-\dot{x}_n(t)\rangle dt  \nonumber\\
&&+ \int_{T_n-d_n}^{T_n} \langle \theta, q_n(t)\rangle dt+ \int_{T_n-d_n}^{T_n} \langle \theta,\dot{x}_n(t) - q_n(t)\rangle dt\nonumber\\
&\le&-h(\alpha,\zeta) d_n + ||\zeta||\int_0^{d_n} ||p_n(t)-\dot{x}_n(t)||dt \nonumber\\
&&\qquad-h(\beta,-\theta)d_n+ ||\theta|| \int_{T_n-d_n}^{T_n} ||\dot{x}_n(t)-q_n(t)|| dt \nonumber\\
&\le& 2d_n +L ||\zeta|| (Md_n^2 + ||\alpha_n-\alpha||d_n) + L||\theta|| (Md_n^2+||\beta_n-\beta||d_n)\nonumber\\
&\le& 2d_n + \kappa \Lambda_n^2, 
\end{eqnarray}
for some constant $\kappa >0$ independent of $n$.

From (\ref{7}), (\ref{12}) and  (\ref{13}), we have
\begin{eqnarray}
\label{14}
C \Lambda_n &<& -T_n + r +\langle (\zeta,\theta),(\alpha_n-\alpha,\beta_n-\beta)\rangle\nonumber\\
&=& -2d_n + \langle (\zeta,\theta),(u_n,w_n)-(\alpha,\beta)\rangle + \langle (\zeta,\theta),(\alpha_n,\beta_n)-(u_n,w_n)\rangle\nonumber\\
&\le& \varepsilon (M||(\zeta,\theta)|| + 2) \Lambda_n + \kappa \Lambda_n^2
\end{eqnarray}
Since $\Lambda_n >0$ for all $n$, it follows from (\ref{14}) that $C < \varepsilon (M||(\zeta,\theta)||+1) + \kappa \Lambda_n$. Letting $n\to \infty$ and then letting $\varepsilon \to 0+$ in the latter inequality, we get $C\le 0$. This is a contradiction.

\textbf{Case 3}. $T_n <r$ for infinitely many $n$. Set $h_n: = (r-T_n)/2$. Let $v\in F(\alpha)$ and $w\in F(\beta)$ be such that $\langle v,\zeta\rangle = \langle w,-\theta\rangle = -1$. Let
$$v_n: = \mathrm{Proj}_{F(\alpha_n)}(v)\quad \text{and}\quad w_n: = \mathrm{Proj}_{F(\beta_n)}(w).$$
Then by the Lipschitz continuity of $F$, we have
\begin{equation}
\label{15}
||v_n-v|| \le L ||\alpha_n-\alpha|| \quad \text{and} \quad ||w_n-w|| \le L ||\beta_n - \beta||.
\end{equation}
For each $n$, by Theorem \ref{C1}, there exist a $C^1$ trajectory $\alpha_n(\cdot)$ of $-F$ and a $C^1$ trajectory $\beta_n(\cdot)$ of $F$ such that
$$\alpha_n(0) = \alpha_n, \dot{\alpha_n}(0) = -v_n, \beta_n(0) = \beta_n \,\,\text{and}\,\,\dot{\beta}_n(0) = w_n,$$
and for some $K>0$,
\begin{equation}
\label{16}
||\dot{\alpha}_n(t) + v_n|| \le Kt,\qquad ||\dot{\beta}_n(t)-w_n|| \le Kt, \,\qquad \forall t\in [0,h_n].
\end{equation}
Set $\gamma_n := \alpha_n(h_n)$ and $\lambda_n: = \beta_n(h_n)$. Then by the triangle inequality,
$$T(\gamma_n,\lambda_n) \le T(\gamma_n,\alpha_n) + T(\alpha_n,\beta_n) + T(\beta_n,\lambda_n) \le h_n + T_n + h_n = r.$$
This means that $(\gamma_n,\lambda_n) \in \mathcal{R}(r)$. Then for any $\varepsilon>0$, for $n$ sufficiently large, we have that
\begin{eqnarray*}
\langle (\zeta,\theta),(\gamma_n,\lambda_n) - (\alpha,\beta)\rangle &\le& \varepsilon || (\gamma_n,\lambda_n) - (\alpha,\beta)|| \nonumber\\
&\le& \varepsilon \left( ||(\gamma_n,\lambda_n) - (\alpha_n,\beta_n)|| + ||(\alpha_n,\beta_n) - (\alpha,\beta)||  \right)\nonumber\\
&\le& \varepsilon \left(||\int_0^{h_n} \dot{\alpha}_n(t)dt|| + ||\int_0^{h_n} \dot{\beta}_n(t)dt|| + \Lambda_n \right)\nonumber\\
&\le& \varepsilon \left(2Mh_n + \Lambda_n   \right). 
\end{eqnarray*}
Thus
\begin{eqnarray}
\label{17}
\langle (\zeta,\theta), (\gamma_n,\lambda_n) - (\alpha_n,\beta_n)\rangle &=& \langle (\zeta,\theta),(\gamma_n,\lambda_n) - (\alpha,\beta)\rangle  + \langle (\zeta,\theta), (\alpha,\beta) - (\alpha_n,\beta_n)\rangle \nonumber\\
&\le&  \varepsilon \left(2Mh_n + \Lambda_n   \right) + ||(\zeta,\theta)||. ||(\alpha_n,\beta_n) - (\alpha,\beta)||\nonumber\\
&=& 2M\varepsilon h_n + \left(\varepsilon + ||(\zeta,\theta)||\right) \Lambda_n.
\end{eqnarray}
We now have that
\begin{eqnarray}
\label{18}
r-T_n &=& 2h_n = -\int_0^{h_n} \langle \zeta,v\rangle dt + \int_0^{h_n} \langle \theta,w\rangle dt \quad (\text{since}\,\langle\zeta,v\rangle = \langle -\theta,w\rangle =-1) \nonumber \\
&=& -\int_0^{h_n} \langle \zeta,v_n\rangle dt + \int_0^{h_n}\langle \zeta,v_n-v\rangle dt + \int_0^{h_n} \langle \theta,w_n\rangle dt +\int_0^{h_n} \langle \theta,w-w_n\rangle dt \nonumber \\
&\le& \int_0^{h_n} \langle \zeta,\dot{\alpha}_n(t)\rangle dt + \int_0^{h_n} \langle \zeta,-\dot{\alpha}_n(t)-v_n\rangle dt  + L||\zeta|| h_n ||\alpha_n - \alpha|| \nonumber\\
&&  +\int_0^{h_n} \langle \theta, \dot{\beta}_n(t)\rangle dt +\int_0^{h_n} \langle \theta,w_n- \dot{\beta}_n(t)\rangle dt + L||\theta||h_n||\beta_n-\beta||  \nonumber \\
&\le& \langle \zeta, \gamma_n - \alpha_n\rangle + K||\zeta|| h_n^2 +\langle \theta, \lambda_n-\beta_n\rangle+ K||\theta||h_n^2 + 2L||(\zeta,\theta)||. ||(\alpha_n-\alpha,\beta_n-\beta)||h_n\nonumber \\
&\le& \langle (\zeta,\theta), (\gamma_n,\lambda_n)-(\alpha_n,\beta_n)\rangle + K||(\zeta,\theta)|| h_n^2 + 2L||(\zeta,\theta)|| \Lambda_n h_n\nonumber\\
&\le& 2M\varepsilon h_n + \left(\varepsilon + ||(\zeta,\theta)||\right) \Lambda_n + K||(\zeta,\theta)|| h_n^2 + 2L||(\zeta,\theta)|| \Lambda_n h_n\nonumber\\
&=& \left[ 2M\varepsilon + K||(\zeta,\theta)||h_n + 2L||(\zeta,\theta)|| \Lambda_n\right] h_n + (\varepsilon + ||(\zeta,\theta)||)\Lambda_n.
\end{eqnarray}
Since $h_n\to 0$, $\Lambda_n \to 0$ as $n\to \infty$ and $\varepsilon >0$ is arbitrary, we can choose $\varepsilon>0$ small enough such that for $n$ sufficiently large, $\left[ 2M\varepsilon + K||(\zeta,\theta)||h_n + 2L||(\zeta,\theta)|| \Lambda_n\right]  < 1$. Then there is some constant $Q>0$ depending only on $\zeta,\theta$ such that for $n$ sufficiently large,
\begin{equation}
\label{19}
h_n \le Q\Lambda_n.
\end{equation}
Now
\begin{eqnarray}
\label{27}
&&T_n - r - \langle (\zeta,\theta),(\alpha_n,\beta_n) - (\alpha,\beta)\rangle \nonumber\\
&&\qquad\qquad =  -2h_n + \langle (\zeta,\theta),(\gamma_n,\lambda_n) - (\alpha_n,\beta_n)\rangle - \varepsilon||(\gamma_n,\lambda_n)-(\alpha,\beta)|| \nonumber\\
&&\qquad\qquad \ge -2h_n + \langle \zeta,\gamma_n - \alpha_n\rangle + \langle \theta,\lambda_n - \beta_n\rangle - \varepsilon||(\gamma_n,\lambda_n)-(\alpha,\beta)|| \nonumber\\
&&\qquad\qquad = -2h_n +\int_0^{h_n} \langle \zeta,\dot{\alpha_n}(t)\rangle dt + \int_0^{h_n} \langle \theta,\dot{\beta}_n(t)\rangle dt -\varepsilon(2Mh_n + \Lambda_n)\nonumber\\
&&\qquad\qquad =\int_0^{h_n} \langle \zeta, v-v_n\rangle dt + \int_0^{h_n} \langle\zeta, \dot{\alpha}_n(t)+v_n\rangle + \int_0^{h_n} \langle \theta,w_n-w\rangle dt\nonumber\\
&&\qquad\qquad\qquad + \int_0^{h_n} \langle \theta,\dot{\beta}_n(t)-w_n\rangle dt -\varepsilon(2Mh_n+\Lambda_n)\nonumber\\
&&\qquad\qquad \ge-L||\zeta||.||\alpha_n-\alpha||h_n - K||\zeta||.||\beta_n - \beta||h_n^2 - L||\theta||h_n - K||\theta||h_n^2 - \varepsilon(2Mh_n + \Lambda_n) \nonumber\\
&&\qquad\qquad \ge -2Q||(\zeta,\theta)||[L+KQ]\Lambda_n^2 -\varepsilon(2MQ+1)\Lambda_n.
\end{eqnarray}
Then, by (\ref{7}) and (\ref{27}),
$$C\Lambda_n < 2Q||(\zeta,\theta)||[L+KQ]\Lambda_n^2 +\varepsilon(2MQ+1)\Lambda_n.$$
Dividing both sides of the latter inequality by $\Lambda_n >0$ then letting $n\to \infty$, we get
$$C\le \varepsilon(2MQ+1).$$
Letting $\varepsilon \to 0+$, we obtain $C\le 0$. This leads to a contradiction. The proof is complete.
\end{proof}
Singular subdifferentials are connected to the non-Lipschitzianity  of a function. It is evident that if the Fr\'echet singular subdifferential of  a function $f$ at a point is nonempty, then $f$ is not Lipschitz around that point. In the next two theorems, we derive formulas for the Fr\'echet singular subdifferentials of the bilateral minimal time function $T$. These formulas may be useful when we study the non-Lipschitz set of $T$. In this paper, the representations of the Fr\'echet singular subdifferentails, together with the representations of the Fr\'echet subdifferentials, are used to study the connection between the Fr\'echet normals to sub-level sets of $T$ and to its epigraph.
\begin{Theorem}
Let $\alpha \in \R^n$. We have
\begin{equation*}
\widehat{\partial}^\infty T(\alpha,\alpha) = \{(\zeta,-\zeta)\in \R^n \times \R^n: h(\alpha,\zeta)\ge 0\}.
\end{equation*}
\end{Theorem}
\begin{proof}
Let $(\zeta,\theta) \in \widehat{\partial}^\infty T(\alpha,\alpha)$. Then for any $\varepsilon >0$, there exists $\sigma >0$ such that
\begin{equation}
\label{21}
\langle (\zeta,\theta), (x,y) - (\alpha,\alpha)\rangle \le \varepsilon (||(x,y)-(\alpha,\alpha)|| + \lambda),
\end{equation}
for all $(x,y)\in B((\alpha,\alpha),\sigma)$ and $\lambda \ge T(x,y)$.

Let $v\in \R^n$. For each $n\in \N^*$, taking $x= y = \alpha + v/n$ and $\lambda = 0$ in (\ref{21}), we have
$$\langle (\zeta,\theta),(v/n,v/n)\rangle \le \varepsilon ||(v/n,v/n)||,$$
or, equivalently,
$$\langle (\zeta,\theta),(v,v)\rangle \le \varepsilon ||(v,v)||.$$
Letting $\varepsilon \to 0+$ in the latter inequality, we get $\langle (\zeta,\theta),(v,v)\rangle \le 0$ for all $v\in \R^n$. This yields $\zeta = -\theta$.

Let $w\in F(\alpha)$ be such that $\langle w,\zeta\rangle = h(\alpha,\zeta)$. By Theorem \ref{C1}, there is a $C^1$ trajectory $x(\cdot)$ 
of $-F$ such that $x(0) = \alpha$ and $\dot{x}(0) = -w$. For $t>0$ sufficiently small, we have $x(t)\in B(\alpha,\sigma)$ and, of course, $T(x(t),\alpha) \le t$. Taking $x = x(t), y = \alpha$ and $\lambda = t$ in (\ref{21}), we  have that
$$\langle \zeta, x(t) - \alpha\rangle \le \varepsilon (||x(t) -\alpha|| + t) ,$$
and then
$$\langle \zeta, x(t) - x(0)\rangle \le \varepsilon (Mt + t).$$
Dividing the latter inequality by $t>0$ then letting $t\to 0+$ and keeping in mind that $\dot{x}(0) = -w$, one gets $\langle \zeta,-w\rangle  \le \varepsilon(M+1)$. Since $\varepsilon >0$ is arbitrary, we conclude that $h(\alpha,\zeta) = \langle \zeta,w\rangle \ge 0$.

Now let $(\alpha,\zeta) \in \R^n \times \R^n$ be such that $h(\alpha,\zeta)\ge 0$. Assume that $(\zeta,-\zeta) \not\in \widehat{\partial}^\infty T(\alpha,\alpha)$, then there exist a constant $C>0$, sequences $\{(\alpha_n,\beta_n)\} \subset \R^n \times \R^n, \{\lambda_n\} \subset \R$ such that $(\alpha_n,\beta_n) \to (\alpha,\alpha)$, $(\alpha_n,\beta_n) \ne (\alpha,\alpha)$, $\lambda_n \ge T(\alpha_n,\beta_n)$  and
$$ \langle (\zeta,-\zeta),(\alpha_n,\beta_n) - (\alpha,\alpha)\rangle  > C(||(\alpha_n,\beta_n) - (\alpha,\alpha)|| + \lambda_n), \quad \forall n.$$
The latter implies that
\begin{equation}
\label{0E}
\langle \zeta, \alpha_n - \beta_n\rangle > C(||(\alpha_n,\beta_n)- (\alpha,\alpha)|| + T(\alpha_n,\beta_n)),\quad \forall n.
\end{equation}
Set $T_n:=T(\alpha_n,\beta_n)$. It follows from (\ref{0E}) that $T_n \le 2||\zeta||. ||\alpha_n - \beta_n||/C <\infty$ for all $n$. Thus, for each $n$, there exists a trajectory $x_n(\cdot)$ of $F$ such that $x_n(0) = \alpha_n, x_n(T_n) = \beta_n$. By Gronwall's Lemma, for all $t\in [0,T_n]$,
$$||x_n(t) - \alpha|| \le ||x_n(t)- \alpha_n|| + ||\alpha_n-\alpha|| \le MT_n + ||(\alpha_n,\beta_n)-(\alpha,\alpha)||.$$
Let $y_n(\cdot) := \mathrm{Proj}_{F(\alpha)}(\dot{x}_n(\cdot))$ on $[0,T_n]$. By the Lipschitz continuity of $F$, 
$$|| y_n(t) - \dot{x}_n(t)|| \le L||x_n(t)-\alpha|| \le LMT_n + L||(\alpha_n,\beta_n)-(\alpha,\alpha)||.$$
Now
\begin{eqnarray}
\label{22}
C(||(\alpha_n,\beta_n)- (\alpha,\alpha)|| + T_n) &\le& \langle \zeta, \alpha_n - \beta_n\rangle \nonumber\\
&\le& \int_0^{T_n} \langle \zeta,\dot{x}_n(t)\rangle dt -\int_0^{T_n}\langle \zeta,y_n(t)\rangle dt\qquad (\text{since}\,\, h(\alpha,\zeta) \ge 0) \nonumber\\
&\le& ||\zeta|| \int_0^{T_n} ||\dot{x}_n(t) - y_n(t)||dt \nonumber\\
&\le&L||\zeta|| (MT_n^2 + ||(\alpha_n,\beta_n)-(\alpha,\alpha)||T_n)\nonumber\\
&\le& C_1 (T_n + ||(\alpha_n,\beta_n)-(\alpha,\alpha)||)^2
\end{eqnarray}
for some constant $C_1>0$ depending only on $\zeta$, $M$ and $L$.

Since $T_n + ||(\alpha_n,\beta_n)-(\alpha,\alpha)|| >0$ for all $n$, it follows from (\ref{22}) that $C \le C_1  (T_n + ||(\alpha_n,\beta_n)-(\alpha,\alpha)||)$ for all $n$. Letting $n\to \infty$ in the latter inequality we get $C\le 0$. This is a contradiction. Thus $(\zeta,-\zeta)\in \widehat{\partial}^\infty T(\alpha,\alpha)$.
\end{proof}

\begin{Theorem} \label{HPN}
For $(\alpha,\beta)\in\mathcal{R}$ with $0<r: = T(\alpha,\beta)$, we have
$$\widehat{\partial}^\infty T(\alpha,\beta) = \widehat{N}_{\mathcal{R}(r)}(\alpha,\beta) \cap \{(\zeta,\theta) \in \R^n \times\R^n : h(\alpha,\zeta) = h(\beta,-\theta) = 0\}.$$
\end{Theorem}
\begin{proof}

Let $(\zeta,\theta) \in \widehat{\partial}^\infty T(\alpha,\beta)$. Then for any $\varepsilon >0$,  there exists $\eta >0$ such that
\begin{equation}\label{E8}
\langle (\zeta,\theta), (x,y) - (\alpha,\beta)\rangle \le \varepsilon \left( ||(x,y) - (\alpha,\beta)|| + |\lambda -T(\alpha,\beta)|\right),
\end{equation}
for all $(x,y) \in B((\alpha,\beta),\eta)$ and $\lambda \ge T(x,y)$.\\
It deduces from (\ref{E8}) that
\begin{equation*}
\langle (\zeta,\theta), (x,y) - (\alpha,\beta)\rangle \le \varepsilon ||(x,y) - (\alpha,\beta)||,
\end{equation*}
for all $(x,y) \in B((\alpha,\beta),\eta) \cap \mathcal{R}(r)$. This means that $(\zeta,\theta) \in \widehat{N}_{\mathcal{R}(r)}(\alpha,\beta)$.

Let $z(\cdot)$ be as in the proof of Proposition \ref{eqH}. By the triangle inequality, we have
$$T(z(t),\beta) \le T(z(t),\alpha) + T(\alpha,\beta) \le t+r,\quad\forall t\in [0,T].$$
In (\ref{E8}), taking $x = z(t),y = \beta$ and $\lambda = t+r$ with $t>0$ sufficiently small, we obtain
$$\langle \zeta,z(t) - \alpha\rangle \le \varepsilon (||z(t) - \alpha|| + t) \le \varepsilon (M+1)t.$$
Dividing the latter inequality by $t>0$ then letting $t\to 0$, we  have
$$-h(\alpha,\zeta) = \langle \zeta,-w\rangle = \langle \zeta,\dot{z}(0)\rangle \le \varepsilon(M+1).$$
Since $\varepsilon >0$ is arbitrary, we conclude that $h(\alpha,\zeta) \ge 0$. Combining with Proposition \ref{eqH}, we obtain $h(\alpha,\zeta) = h(\beta,-\theta) = 0$.

Now let $(\zeta,\theta) \in \widehat{N}_{\mathcal{R}(r)}(\alpha,\beta)$ with $h(\alpha,\zeta) = h(\beta,-\theta) = 0$. We will show that $(\zeta,\theta) \in \widehat{\partial}^\infty T(\alpha,\beta)$. Assume to the contrary that $(\zeta,\theta)\not\in \widehat{\partial}^\infty T(\alpha,\beta)$, then there exist a constant $C>0$ and sequences $\{(\alpha_n,\beta_n)\} \subset \R^n \times \R^n$, $\{\lambda_n\}\subset \R$ such that $(\alpha_n,\beta_n) \to (\alpha,\beta)$, $(\alpha_n,\beta_n) \ne (\alpha,\beta)$, $\lambda_n \ge T_n: = T(\alpha_n,\beta_n)$ and
\begin{equation}
\label{E9}
\langle (\zeta,\theta), (\alpha_n,\beta_n) - (\alpha,\beta)\rangle > C \left( ||(\alpha_n,\beta_n) - (\alpha,\beta)|| + | \lambda_n -T(\alpha,\beta)|\right),\quad \forall n.
\end{equation}
 There are two cases.\\

\textbf{Case 1}. $T_n \le r$ for infinitely many $n$. In this case $(\alpha_n,\beta_n) \in \mathcal{R}(r)$. Since $(\zeta,\theta) \in \widehat{N}_{\mathcal{R}(r)}(\alpha,\beta)$, for any $\varepsilon >0$, there is a number $n_0\in \N$ such that for all $n \ge n_0$, we have
$$\langle (\zeta,\theta), (\alpha_n,\beta_n) - (\alpha,\beta)\rangle \le \varepsilon  ||(\alpha_n,\beta_n) - (\alpha,\beta)||.$$
Combining with (\ref{E9}) we have, for $n$ sufficiently large, that
$$C||(\alpha_n,\beta_n) - (\alpha,\beta)|| \le C \left( ||(\alpha_n,\beta_n) - (\alpha,\beta)|| + | \lambda_n -T(\alpha,\beta)|\right) <\varepsilon  ||(\alpha_n,\beta_n) - (\alpha,\beta)||.$$
Since $ ||(\alpha_n,\beta_n) - (\alpha,\beta)|| >0$, it follows from the latter inequalities that $C <\varepsilon$. This is a contradiction since $\varepsilon >0$ is arbitrary.

\textbf{Case 2}. $T_n >r$ for infinitely many $n$. Set  $h_n : = (T_n-r)/2$. Let $x_n(\cdot)$ be a trajectory of $F$ such that $x_n(0) =\alpha_n$ and $x_n(T_n) =\beta_n$. Set $\gamma_n: = x_n(h_n), w_n:= x_n(h_n+r) = x_n(T_n-h_n)$. Then $T(\gamma_n,w_n) = r$ and therefore $(\gamma_n,w_n) \in \mathcal{R}(r)$.\\
For $t\in [0,T_n]$, we have
\begin{equation}
\label{E10}
||x_n(t)-\alpha|| \le ||x_n(t) - \alpha_n|| + ||\alpha_n-\alpha|| \le Mt + ||\alpha_n-\alpha||,
\end{equation}
and
\begin{equation}
\label{E11}
||x_n(T_n-t)-\beta|| \le ||x_n(T_n-t) - \beta_n|| + ||\beta_n-\beta|| \le Mt + ||\beta_n-\beta||.
\end{equation}
Let $p_n(\cdot)$ and $q_n(\cdot)$ be the projections of $\dot{x}_n(\cdot)$ on $F(\alpha)$ and $F(\beta)$, respectively, restricted to $[0,T_n]$. From (\ref{E10}), (\ref{E11}) and by Lipschitz continuity of $F$ we have
\begin{equation}
\label{E12}
||\dot{x}_n(t) -p_n(t)|| \le L||x_n(t) - \alpha|| \le LMt + L||\alpha_n-\alpha||,
\end{equation} 
and
\begin{equation}
\label{E13}
||\dot{x}_n(T_n-t) -q_n(T_n-t)|| \le L||x_n(T_n-t) - \beta|| \le LMt + L||\beta_n-\beta||,
\end{equation}
for all $t\in [0,T_n]$.

Now
\begin{equation}
\label{E14}
\langle (\zeta,\theta),(\alpha_n,\beta_n) -(\alpha,\beta) \rangle = \langle (\zeta,\theta),(\alpha_n,\beta_n) -(\gamma_n,w_n) \rangle + \langle (\zeta,\theta),(\gamma_n,w_n) -(\alpha,\beta) \rangle.
\end{equation}

We first estimate the second term on the right-hand side of (\ref{E14}). Since $(\gamma_n,w_n) \in \mathcal{R}(r)$ and $(\zeta,\theta) \in \widehat{N}_{\mathcal{R}(r)}(\alpha,\beta)$, for any $\varepsilon >0$ and for all $n$ sufficiently large, one has
\begin{equation*}
\langle (\zeta,\theta),(\gamma_n,w_n) - (\alpha,\beta)\rangle \le \varepsilon ||(\gamma_n,w_n) - (\alpha,\beta)||.
\end{equation*}
Using (\ref{E10}) - (\ref{E11}), 
\begin{eqnarray}
\label{E15}
\langle (\zeta,\theta),(\gamma_n,w_n) - (\alpha,\beta)\rangle &\le& \varepsilon \{ (Mh_n + ||\alpha_n-\alpha||) + (Mh_n+||\beta_n-\beta||) \} \nonumber\\
&\le& 2(M+1)\varepsilon(||(\alpha_n,\beta_n) - (\alpha,\beta)|| +2h_n).
\end{eqnarray}
Using (\ref{E12}) - (\ref{E13}) and the fact $h(\alpha,\zeta) = h(\beta,-\theta)=0$, we can estimate the first term on the right-hand side of (\ref{E14}) as follows
\begin{eqnarray}
\label{E16}
\langle (\zeta,\theta),(\alpha_n,\beta_n) -(\gamma_n,w_n) \rangle &=& \langle \zeta,\alpha_n-\gamma_n \rangle + \langle \theta,\beta_n-w_n\rangle \nonumber \\
&=& -\int_0^{h_n} \langle\zeta,\dot{x}_n(s)\rangle ds + \int_{T_n-h_n}^{T_n} \langle \theta,\dot{x}_n(s)\rangle ds\nonumber\\
&=& -\int_0^{h_n} \langle\zeta,p_n(s)\rangle ds + \int_0^{h_n} \langle\zeta,p_n(s)-\dot{x}_n(s)\rangle ds \nonumber\\
&&+ \int_{T_n-h_n}^{T_n} \langle \theta,q_n(s)\rangle ds + \int_{T_n-h_n}^{T_n} \langle \theta,\dot{x}_n(s)-q_n(s)\rangle ds\nonumber\\
&\le& -h(\alpha,\zeta)h_n + ||\zeta||\int_0^{h_n} ||p_n(s)-\dot{x}_n(s)||ds \nonumber\\
&&-h(\beta,-\theta)h_n + ||\theta||\int_{T_n-h_n}^{T_n}||\dot{x}_n(s) - q_n(s)|| ds\nonumber\\
&\le& ||\zeta||(LMh_n^2 + Lh_n||\alpha_n-\alpha||) + ||\theta||(LMh_n^2 + Lh_n||\beta_n-\beta||)\nonumber\\
&\le& C_0(||(\alpha_n,\beta_n) - (\alpha,\beta)||+ 2h_n)^2,
\end{eqnarray}
for some suitable constant $C_0>0$.\\
From (\ref{E14}) -(\ref{E16}), we have
\begin{equation*}
\langle (\zeta,\theta),(\alpha_n,\beta_n) -(\alpha,\beta) \rangle \le 2(M+1)\varepsilon (||(\alpha_n,\beta_n) - (\alpha,\beta)|| +2h_n) + C_0(||(\alpha_n,\beta_n) - (\alpha,\beta)|| +2h_n)^2.
\end{equation*}
Combining with (\ref{E9}), we have that
\begin{eqnarray*}
C \left( ||(\alpha_n,\beta_n) - (\alpha,\beta)|| + 2h_n\right) &\le& C \left( ||(\alpha_n,\beta_n) - (\alpha,\beta)|| + | \lambda_n -T(\alpha,\beta)|\right)\\
&<& 2(M+1)\varepsilon (||(\alpha_n,\beta_n) - (\alpha,\beta)|| +2h_n) \\
&&+ \,C_0(||(\alpha_n,\beta_n) - (\alpha,\beta)|| +2h_n)^2.
\end{eqnarray*}
It follows that 
$$C < 2(M+1)\varepsilon + C_0(||(\alpha_n,\beta_n) - (\alpha,\beta)|| +2h_n).$$
Letting $n\to \infty$ and then letting $\varepsilon \to 0+$ in the latter inequality, we get $C\le 0$. This contradiction ends the proof. 
\end{proof}

The next theorem gives a connection between Fr\'echet normals to sub-level sets and to the epigraph of the bilateral minimal time function.
\begin{Theorem} 
\label{RE}
Let $(\alpha,\beta) \in \mathcal{R}$ be such that $0 <r:=T(\alpha,\beta)$.
\item[(i)] If $(\zeta,\theta) \in \widehat{N}_{\mathcal{R}(r)}(\alpha,\beta)$ then $h(\alpha,\zeta) = h(\beta,-\theta)$ and $((\zeta,\theta),h(\alpha,\zeta)) \in \widehat{N}_{\mathrm{epi}(T)}((\alpha,\beta),r)$.
\item[(ii)] If $(\alpha,\beta) \in \R^n\times \R^n$ and $\lambda \in \R$ satisfy $((\zeta,\theta),\lambda) \in \widehat{N}_{\mathrm{epi}(T)}((\alpha,\beta),r)$, then $\lambda \le 0$, $h(\alpha,\zeta) = h(\beta,-\theta) = \lambda$ and $(\zeta,\theta) \in \widehat{N}_{\mathcal{R}(r)}(\alpha,\beta)$.
\end{Theorem}
\begin{proof}
(i) Since $(\zeta,\theta) \in \widehat{N}_{\mathcal{R}(r)}(\alpha,\beta)$, it follows from Proposition \ref{eqH} that $h(\alpha,\zeta) = h(\beta,-\theta) \le 0$. We have two cases.\\
\textbf{Case 1}. $h(\alpha,\zeta) = h(\beta,-\theta) = 0$. Then by Theorem \ref{HPN}, $(\zeta,\theta) \in \widehat{\partial}^\infty T(\alpha,\beta)$. Equivalently, $((\zeta,\theta),0) \in \widehat{N}_{\mathrm{epi}(T)}((\alpha,\beta),r)$. Thus $((\zeta,\theta),h(\alpha,\zeta)) \in \widehat{N}_{\mathrm{epi}(T)}((\alpha,\beta),r)$.
\\\textbf{Case 2}. $\lambda: =h(\alpha,\zeta) = h(\beta,-\theta) <0$. We set
$$\zeta_1 = -\frac{\zeta}{\lambda},\,\,\text{and}\,\, \theta_1 = - \frac{\theta}{\lambda}.$$
Then $(\zeta_1,\theta_1) \in \widehat{N}_{\mathcal{R}(r)}(\alpha,\beta)$ and $h(\alpha,\zeta_1) = h(\beta,-\theta_1) = -1$. Then by Theorem \ref{PN}, we have $(\zeta_1,\theta_1) \in \widehat{\partial} T(\alpha,\beta)$. Equivalently, $((\zeta_1,\theta_1),-1) \in \widehat{N}_{\mathrm{epi}(T)}((\alpha,\beta),r)$. Therefore
$$  ((\zeta,\theta),h(\alpha,\zeta)) = -\lambda((\zeta_1,\theta_1),-1) \in  \widehat{N}_{\mathrm{epi}(T)}((\alpha,\beta),r).$$
(ii) By the nature of an epigraph, it follows from $((\zeta,\theta),\lambda) \in \widehat{N}_{\mathrm{epi}(T)}((\alpha,\beta),r)$ that $\lambda \le 0$. We also have two possible cases.\\
\textbf{Case 1}. $\lambda = 0$. Then $(\zeta,\theta) \in \widehat{\partial}^\infty T(\alpha,\beta)$. By Theorem \ref{HPN}, we have
$$(\zeta,\theta) \in \widehat{N}_{\mathcal{R}(r)}(\alpha,\beta),\,\,\text{and}\,\,h(\alpha,\zeta) = h(\beta,-\theta) = 0 = \lambda.$$
\textbf{Case 2}. $\lambda <0$. Set
$$\zeta_1 = -\frac{\zeta}{\lambda},\,\,\text{and}\,\, \theta_1 = - \frac{\theta}{\lambda}.$$
Then $((\zeta_1,\theta_1),-1)\in \widehat{N}_{\mathrm{epi}(T)}((\alpha,\beta),r)$. This implies $(\zeta_1,\theta_1) \in \widehat{\partial} T(\alpha,\beta)$. By Theorem \ref{PN},
$$(\zeta_1,\theta_1) \in \widehat{N}_{\mathcal{R}(r)}(\alpha,\beta),\,\,\text{and}\,\, h(\alpha,\zeta_1) = h(\beta,-\theta_1) = -1.$$
Thus $(\zeta,\theta) = -\lambda (\zeta_1,\theta_1) \in \widehat{N}_{\mathcal{R}(r)}(\alpha,\beta)$ and $h(\alpha,\zeta) = h(\beta,-\theta) = \lambda$.
\end{proof}
The result in Theorem \ref{RE} can be stated in the following way
\begin{Theorem}\label{RE1}
Let $(\alpha,\beta) \in \mathcal{R}$ be such that $0<r:=T(\alpha,\beta)$. We have $(\zeta,\theta)\in \widehat{N}_{\mathcal{R}(r)}(\alpha,\beta)$  if and only if
$$((\zeta,\theta),h(\alpha,\zeta)) =  ((\zeta,\theta),h(\beta, -\theta))  \in \widehat{N}_{\mathrm{epi}(T)}((\alpha,\beta),r).$$
\end{Theorem}
Using Theorem \ref{RE1}, one can easily prove the following Proposition
\begin{Proposition}\label{ZN}
Let $(\alpha,\beta) \in \mathcal{R}$ with $0<T(\alpha, \beta)$. One has 
$$\widehat{N}_{\mathcal{R}(T(\alpha,\beta))}(\alpha,\beta)  = \{(0,0)\}\,\,\text{if and only if}\,\, 
\widehat{N}_{\mathrm{epi}(T)}((\alpha,\beta),T(\alpha,\beta)) = \{((0,0),0)\}.$$
\end{Proposition}
The following result is a special feature of the bilateral minimal time function.
\begin{Theorem}
\label{dim}
Let $(\alpha,\beta) \in \mathcal{R}$ with $\alpha \ne \beta$. We have
\begin{equation}
\label{edim}
\dim \widehat{N}_{\mathcal{R}(T(\alpha,\beta))}(\alpha,\beta) = \dim \widehat{N}_{\mathrm{epi}(T)}((\alpha,\beta),T(\alpha,\beta)).
\end{equation}
\end{Theorem}
\begin{proof} The argument is close to the one in \cite{LN} where an analogous result is proved for proximal normal cones in the context of the unilateral minimal time function.
By Proposition \ref{ZN}, it is sufficient to consider the case when both $\widehat{N}_{\mathcal{R}(T(\alpha,\beta))}(\alpha,\beta)$ and $\widehat{N}_{\mathrm{epi}(T)}((\alpha,\beta),T(\alpha,\beta))$ are nontrivial. \\
Set $r = T(\alpha,\beta)$. Let $\kappa = \dim \widehat{N}_{\mathcal{R}(r)}(\alpha,\beta)$ and $\ell = \dim \widehat{N}_{\mathrm{epi}(T)}((\alpha,\beta),r)$. We first assume that $(\zeta_1,\theta_1),\cdots, (\zeta_\kappa,\theta_\kappa) \in \widehat{N}_{\mathcal{R}(r)}(\alpha,\beta)$ are linearly independent. By Theorem
\ref{RE1}, we have $$((\zeta_1,\theta_1),h(\alpha,\zeta_1)), \cdots, ((\zeta_\kappa,\theta_\kappa),h(\alpha,\zeta_\kappa)) \in \widehat{N}_{\mathrm{epi}(T)}((\alpha,\beta),r).$$ Observe that $((\zeta_1,\theta_1),h(\alpha,\zeta_1)), \cdots, ((\zeta_\kappa,\theta_\kappa),h(\alpha,\zeta_\kappa))$ are linearly independent. Thus $\kappa \le \ell$.\\
Now assume that $((\zeta_1,\theta_1),\lambda_1),\cdots, (\zeta_\ell,\theta_\ell),\lambda_\ell) \in \widehat{N}_{\mathrm{epi}(T)}((\alpha,\beta),r)$ are linearly independent. It follows from Theorem \ref{RE} that $(\zeta_i,\theta_i) \in \widehat{N}_{\mathcal{R}(r)}(\alpha,\beta)$ and $h(\alpha,\zeta_i) = h(\beta,-\theta_i) = \lambda_i$ for all $i = 1,\cdots,\ell$. Observe that $(\zeta_i,\theta_i)\ne (0,0)$ for all $i = 1,\cdots,\ell$. Indeed, if $(\zeta_i,\theta_i) = (0,0)$ for some $i$, then $\lambda_i =0$. This contradicts to the linear independence of $((\zeta_1,\theta_1),\lambda_1),\cdots, (\zeta_\ell,\theta_\ell),\lambda_\ell)$. We are going to show that $(\zeta_1,\theta_1),\cdots, (\zeta_\ell,\theta_\ell)$ are linearly independent. Consider
\begin{equation}
\label{E17}
\sum_{i=1}^\ell a_i(\zeta_i,\theta_i) = 0,
\end{equation}
for some $a_1,\cdots,a_\ell \in \R$. We claim that $a_1 = \cdots = a_\ell = 0$. Indeed, set 
$$\mathcal{I} = \{i:  a_i \ge 0, 1\le i \le \ell\},$$
and
$$\mathcal{J} = \{i:  a_i <0, 1\le i \le \ell\}.$$
Then $\mathcal{I} \cup \mathcal{J} = \{1,\cdots,\ell\}$.\\
Assume  $\mathcal{I}\ne \emptyset$ and $\mathcal{J} \ne \emptyset$. From (\ref{E17}), we have
$$\sum_{i\in\mathcal{I}} a_i(\zeta_i,\theta_i) = -\sum_{j\in\mathcal{J}}a_j(\zeta_j,\theta_j),$$
i.e.,
\begin{equation}
\label{E18}
\sum_{i\in\mathcal{I}}a_i\zeta_i = -\sum_{j\in\mathcal{J}}a_j\zeta_j\,\,\text{and}\,\, \sum_{i\in\mathcal{I}}a_i\theta_i = -\sum_{j\in\mathcal{J}}a_j\theta_j
\end{equation}
Since $((\zeta_i,\theta_i),\lambda_i) \in \widehat{N}_{\mathrm{epi}(T)}((\alpha,\beta),r),$ and  $a_i \ge 0$ for all $i\in \mathcal{I}$, we have
$$\left(\left(\sum_{i\in \mathcal{I}}a_i\zeta_i,\sum_{i\in \mathcal{I}}a_i\theta_i\right),\sum_{i\in \mathcal{I}}a_i\lambda_i\right) = \sum_{i\in \mathcal{I}}a_i((\zeta_i,\theta_i),\lambda_i) \in \widehat{N}_{\mathrm{epi}(T)}((\alpha,\beta),r).$$ 
Then by Theorem \ref{RE}, 
\begin{equation}
\label{E19}
h\left(\alpha, \sum_{i\in \mathcal{I}}a_i\zeta_i\right) = \sum_{i\in \mathcal{I}}a_i\lambda_i.
\end{equation}
Similarly, since $((\zeta_j,\theta_j),\lambda_j) \in \widehat{N}_{\mathrm{epi}(T)}((\alpha,\beta),r)$, and $ -a_j \ge 0$ for all $j\in \mathcal{J}$, we have
$$\left(\left(-\sum_{j\in \mathcal{J}}a_j\zeta_j,-\sum_{j\in \mathcal{J}}a_j\theta_j\right),-\sum_{j\in \mathcal{J}}a_j\lambda_j\right) = -\sum_{j\in \mathcal{J}}a_j((\zeta_j,\theta_j),\lambda_j) \in \widehat{N}_{\mathrm{epi}(T)}((\alpha,\beta),r).$$ 
Then
\begin{equation}
\label{E20}
h\left(\alpha, -\sum_{j\in \mathcal{J}}a_j\zeta_j\right) = -\sum_{j\in \mathcal{J}}a_j\lambda_j.
\end{equation}
It follows from (\ref{E18}) - (\ref{E20}) that
$$    \sum_{i\in \mathcal{I}}a_i\lambda_i  =-\sum_{j\in \mathcal{J}}a_j\lambda_j,$$
i.e.,
\begin{equation}
\label{E21}
\sum_{i=1}^\ell a_i\lambda_i = 0.
\end{equation}
We have from (\ref{E17}) and (\ref{E21}) that
$$\sum_{i=1}^\ell a_i((\zeta_i,\theta_i),\lambda_i) = ((0,0),0).$$
Since $((\zeta_1,\theta_1),\lambda_1),\cdots, ((\zeta_\ell,\theta_\ell),\lambda_\ell)$ are linearly independent, the latter equality implies that $a_1 = \cdots = a_\ell=0$. This contradicts to $\mathcal{J} \ne \emptyset$.\\ 
Similarly, it cannot happen that $\mathcal{I} = \emptyset$ and $\mathcal{J}\ne \emptyset$. In the case, $\mathcal{I} \ne \emptyset$ and $\mathcal{J} = \emptyset$, we get $a_1 = \cdots = a_\ell = 0$. This means that $(\zeta_1,\theta_1),\cdots, (\zeta_\ell,\theta_\ell)$ are linearly independent. Thus $\ell \le \kappa$. This ends the proof.
\end{proof}
\begin{Remark}
It is worth remarking that all results in this paper still hold true if we replace  Fr\'echet normal cones, Fr\'echet (singular) subdifferentials by proximal normal cones, proximal (singular) subdifferentials, respectively.
\end{Remark} 
\textbf{Acknowledgments}. The author would like to thank the anonymous referee  for valuable remarks and suggestions  which improved the quality of the paper.

The paper was supported by funds allocated to the implementation of the international co-funded project in the years 2014-2018, 3038/7.PR/2014/2, and by the EU grant PCOFUND-GA-2012-600415.


\begin{thebibliography}{999999}
\bibitem{CS95} P. Cannarsa, C. Sinestrari, Convexity properties of the minimun time function. Calc. Var. Partial Differential Equations 3 (1995) 273-298.

\bibitem{CSB04} P. Cannarsa and C. Sinestrari, Semiconcave Functions, Hamilton-Jacobi Equations and Optimal Control. Birkhauser, Boston (2004).

\bibitem{CAK} P. Cannarsa, K. T. Nguyen, Exterior sphere condition and time optimal control for differential inclusions. SIAM J. Control Optim. 49 (2011) 2558--2576.

\bibitem{CLSW}  F. H. Clarke, Y. S. Ledyaev, R. J. Stern,  P. R. Wolenski,  Nonsmooth Analysis and Control Theory, Springer, New York (1998).

\bibitem{CN04} F. H. Clarke, C. Nour, The Hamilton - Jacobi equation of minimal time control. J. Convex Anal. 11 (2004) 413--436.

\bibitem{CMW} G. Colombo, A. Marigonda, P. Wolenski, Some new regularity properties for the minimal time function, SIAM J. Control Optim. 44 (2006) 2285--2299.

\bibitem{GK}  G. Colombo, K. T. Nguyen, On the structure of the minimum time function. SIAM J. Control Optim. 48 (2010) 4776--4814.

\bibitem{CKL} G. Colombo, K. T. Nguyen, L. V.  Nguyen, Luong, Non-Lipschitz points and the SBV regularity of the minimum time function. Calc. Var. Partial Differential Equations 51 (2014) 439--463.

\bibitem{CL} G. Colombo, L. V. Nguyen, Differentiability properties of the minimum time function for normal linear systems. J. Math. Anal. Appl. 429 (2015) 143--174.

\bibitem{CW1} G, Colombo, P. R. Wolenski, Variational analysis for a class of minimal time functions in Hilbert spaces. J. Convex Anal. 11 (2004), 335--361.

\bibitem{CW2} G, Colombo, P. R. Wolenski, The subgradient formula for the minimal time function in the case of constant dynamics in Hilbert space. J. Global Optim. 28 (2004) 269--282.

\bibitem{CN06} C. Nour, The bilateral minimal time function, J. of Convex Anal. 13 (2006), 61-80.

\bibitem{CN13} C. Nour, Proximal subdifferential of the bilateral minimal time function and some regularity applications,   J. Convex Anal. 20 (2013) 1095--1112.

\bibitem{NS1} C. Nour, R. J. Stern, Semiconcavity of the bilateral minimal time function: the linear case. Systems Control Lett. 57 (2008) 863--866.

\bibitem{NS2} C. Nour, R.J. Stern, The state constrained bilateral minimal time function. Nonlinear Anal. 69 (2008) 3549--3558. 

\bibitem{NS3} C. Nour, R.J. Stern, Regularity of the state constrained minimal time function, Nonlinear Anal. 66  (2007) 62--72.

\bibitem{LN} L. V. Nguyen, Variational analysis and sensitivity relations for the minimum time function. SIAM J. Control Optim. 54 (2016) 2235-2258.

\bibitem{RW} {\sc R. T. Rockafellar \& R. J-B. Wets}, {\em Variational Analysis}, Springer, Berlin (1998).

\bibitem{S04} R.J. Stern, Characterization of the state constrained minimal time function, SIAM J. Control Optim. 43 (2004) 697--707.

\bibitem{WY98} P. R. Wolenski, Y. Zhuang, Proximal analysis and the minimal time function. SIAM J. Control Optim. 36 (1998) 1048 -1072.


\end{thebibliography}
\end{document}